\documentclass{article}

\usepackage[utf8]{inputenc} 
\usepackage[T1]{fontenc}    
\usepackage{hyperref}       
\usepackage{url}            
\usepackage{booktabs}       
\usepackage{nicefrac}       
\usepackage{microtype}      
\usepackage{graphicx}
\usepackage{authblk}

\usepackage{amsmath}
\usepackage{fourier}
\usepackage{braket}
\usepackage[standard,amsmath,thmmarks]{ntheorem}

\usepackage{algorithm}
\usepackage{algpseudocode} 

\usepackage{bbold}

\usepackage{etoolbox}

\usepackage{chngcntr}
\usepackage{apptools}
\AtAppendix{\counterwithin{lemma}{section}}
\AtAppendix{\counterwithin{corollary}{section}}

\usepackage{ifplatform}

\newcommand{\norm}[1]{\Vert #1 \Vert}
\newcommand{\abs}[1]{\vert #1 \vert}

\newcommand{\eu}{\mathrm{e}}

\DeclareMathOperator{\dom}{\mathrm{dom}}

\DeclareMathOperator{\inte}{\mathrm{int}}
\DeclareMathOperator{\argmin}{ \arg \min }

\DeclareMathOperator{\tr}{\mathrm{Tr}}

\begin{document}

\title{Convergence of the Exponentiated Gradient Method with Armijo Line Search}
\author{Yen-Huan~Li}
\author{Volkan~Cevher}
\affil{Laboratory for Information and Inference Systems \\ \'{E}cole Polytechnique F\'{e}d\'{e}rale de Lausanne}
\date{}

\maketitle

\newtoggle{qqed}
\togglefalse{qqed}

\newtoggle{kkeyword}
\togglefalse{kkeyword}

\begin{abstract}
Consider the problem of minimizing a convex differentiable function on the probability simplex, spectrahedron, or set of quantum density matrices. 
We prove that the exponentiated gradient method with Armjo line search always converges to the optimum, if the sequence of the iterates possesses a strictly positive limit point (element-wise for the vector case, and with respect to the L\"{o}wner partial ordering for the matrix case). 
To the best our knowledge, this is the first convergence result for a mirror descent-type method that only requires differentiability. 
The proof exploits self-concordant likeness of the log-partition function, which is of independent interest. 
\end{abstract}

\iftoggle{kkeyword}{
}{}

\section{Introduction}

Consider the optimization problem
\begin{equation}
f^\star = \min \set{ f ( \rho ) | \rho \in \mathcal{D} } , \tag{P} \label{eq_problem}
\end{equation}
where $f$ is a convex function differentiable on $\inte \dom f$, and $\mathcal{D}$ denotes the set of quantum density matrices, i.e., 
\begin{equation}
\mathcal{D} := \set{ \rho \in \mathbb{C}^{d \times d} | \rho \geq 0, \tr \rho = 1 } , \notag
\end{equation}
for some positive integer $d$. 
We assume that $f^\star > - \infty$. 

This problem formulation \eqref{eq_problem} allows us to address two other constraints simultaneously: 
\begin{itemize}
\item The probability simplex $\mathcal{P} := \set{ x \in \mathbb{R}_+^d | \norm{ x }_1 = 1 }$. 
\item The spectrahedron $\mathcal{S} := \set{ X \in \mathbb{R}^{d \times d} | X \geq 0, \tr X = 1 }$. 
\end{itemize}
Optimization problems with a probability simplex, spectrahedron, or quantum density matrix constraint appear in various applications, such as sparse regression \cite{Tibshirani1996}, low-rank matrix estimation \cite{Koltchinskii2011b}, and quantum state tomography \cite{Paris2004}, to mention a few; the corresponding objective functions are typically convex and differentiable. 

Starting with some non-singular $\rho_0 \in \mathcal{D}$, the exponentiated gradient  (EG) method iterates as
\begin{equation}
\rho_k = C_k^{-1} \exp \left[ \log ( \rho_{k - 1} ) - \alpha_k \nabla f ( \rho_{k - 1} ) \right] , \quad k \in \mathbb{N} , \label{eq_eg}
\end{equation}
where $C_k$ is a positive real number normalizing the trace of $\rho_k$, and $\alpha_k > 0$  denotes the step size. 
Equivalently, one may write
\begin{equation}
\rho_k \in \argmin \set{ \alpha_k \braket{ \nabla f ( \rho_{k - 1}, \sigma - \rho_{k - 1} ) } + D ( \sigma, \rho_{k - 1} ) | \sigma \in \mathcal{D} } , \label{eq_eg_equivalent}
\end{equation}
where $D$ denotes the quantum relative entropy. 
Therefore, the EG method can be viewed as entropic mirror descent without averaging \cite{Beck2003,Nemirovsky1983}, or a special case of the interior gradient method \cite{Auslender2006}. 

\begin{algorithm}[t]

\caption{Exponentiated Gradient Method with Armijo Line Search}

\label{alg}

\begin{algorithmic}[1]
\Require 
$\bar{\alpha} > 0$, $r \in ( 0, 1 )$, $\tau \in ( 0, 1 )$, $\rho_0 \in \mathcal{D}$ non-singular
\For{$k = 1, 2, \ldots$}
\State $\alpha_k \leftarrow \bar{\alpha}$
\While{$f( \rho_{k - 1} ( \alpha_k ) ) > f ( \rho_{k - 1} ) + \tau \Braket{ f' ( \rho_{k - 1} ), \rho_{k - 1} ( \alpha_k ) - \rho_{k - 1} } $}
\State $\alpha_k \leftarrow r \alpha_k$
\EndWhile
\State $\rho_{k} \leftarrow \rho_{k - 1} ( \alpha_k )$
\EndFor
\end{algorithmic}

\end{algorithm}

There are various approaches to selecting the step size. 
In this paper, we will focus on Armijo line search. 
Let $\bar{\alpha} > 0$ and $r, \tau \in ( 0, 1 )$. 
The Armijo line search procedure outputs $\alpha_k = r^j \bar{\alpha}$, where $j$ is the least non-negative integer that satisfies
\begin{equation}
f ( \rho_k ) \leq f ( \rho_{k - 1} ) + \tau \braket{ \nabla f ( \rho_{k - 1} ), \rho_k - \rho_{k - 1} } ; \notag
\end{equation}
the dependence on $j$ lies implicitly in $\rho_k$. 
We give the pseudo codes in Algorithm \ref{alg}, where we define
\begin{equation}
\rho_{k - 1} ( \alpha_k ) := \tilde{C}_{k}^{-1} \exp \left[ \log ( \rho_{k - 1} ) - \alpha_k \nabla f ( \rho_{k - 1} ) \right] , \quad \forall k \in \mathbb{N} ; \notag
\end{equation}
$\tilde{C}_k$ normalizes the trace of $\rho_{k - 1} ( \alpha_k )$. 

Implementing Armijo line search does not require any parameter of the objective function, e.g., the Lipschitz constant of the objective function or its gradient. 
This observation shows the possibility of proving a convergence guarantee for the EG method with respect to a general class of objective functions. 
Indeed, we will only assume that the objective function is convex and differentiable throughout this paper. 

\subsection{Motivation and Related Work}

Regarding the structure of the constraint set, the EG method is a natural choice among mirror descent-type methods. 
Especially, for the vector case where the constraint set is the probability simplex, the iteration rule becomes computationally cheap---projection is not required. 
However, existing convergence guarantees for the EG method imposes restrictive conditions on the objective function. 

We summarize briefly existing convergence results for the EG method. 
If $f$ is $L$-Lipschitz continuous, standard analysis of mirror descent shows that the EG method converges to the optimum with averaged iterates \cite{Beck2003}. 
If $\nabla f$ is $L$-Lipschitz continuous, the EG method converges either with a constant step size or Armijo line search \cite{Auslender2006}. 
Recently, the Lipschitz gradient condition was generalized by the notion of relative smoothness \cite{Bauschke2017,Lu2016}. 
We say that $f$ is $L$-smooth relative to a convex function $h$, if $L h - f$ is a convex function. 
If $f$ is $L$-smooth relative to the negative von Neumann entropy, the EG method converges with a constant step size \cite{Bauschke2017,Collins2008,Lu2016}. 

Notice that checking the conditions can be highly non-trivial, and there are applications where none of the conditions above hold. 
One such application is quantum state tomography \cite{Paris2004}. 
Quantum state tomography corresponds to solving \eqref{eq_problem} with the objective function 
\begin{equation}
f_{\text{QST}} ( \rho ) := - \sum_{i = 1}^n \log \tr ( M_i \rho ) , \notag
\end{equation}
where $M_i$ are Hermitian positive semi-definite matrices given by the experimental data. 

\begin{proposition} \label{prop_qst_is_hard}
The function $f_{\text{QST}}$ is not Lipschitz, its gradient is not Lipschitz, and it is not smooth relative to the negative von Neumann entropy. 
\end{proposition}

The proof of Proposition \ref{prop_qst_is_hard} can be found in Section \ref{sec_qst_is_hard}. 
Similar loss functions also appear for the cases of probability simplex and spectrahedron constraints, such as positive linear inverse problems, positron emission tomography, portfolio selection, and Poisson phase retrieval \cite{Byrne2001,MacLean2012,Odor2016,Vardi1985}. 

There are indeed convergence guarantees that require mild differentiability conditions, though they are all for gradient descent-type methods. 
Bertsekas proved that the projected gradient descent with Armijo line search always converges for a differentiable objective function, when the constraint is a box or the positive orthant \cite{Bertsekas1976}. 
Gafni and Bertsekas generalized the previous result for any compact convex constraint \cite{Gafni1982}. 
Salzo proved the convergence of proximal variable metric methods with various line search schemes, assuming that $\nabla f$ is uniformly continuous on any compact set \cite{Salzo2017}. 

In comparison to existing results, we highlight the following contributions. 
\begin{itemize}
\item To the best of our knowledge, we give the first convergence guarantee of a mirror descent-type method\footnote{Here we exclude the very standard projected gradient method.} that only requires differentiability. 
\item Our convergence analysis exploits the self-concordant likeness of the log partition function. 
As a by-product, we improve on the Peierls-Bogoliubov inequality, which is of independent interest (cf. Remark \ref{rem_peierls_bogoliubov}). 
\end{itemize}

\subsection{Main Result}

Our main result is the following theorem. 

\begin{theorem} \label{thm_main}
Suppose that $f$ is differentiable at every non-singular $\rho \in \mathcal{D}$. 
Then we have: 
\begin{enumerate}
\item The Armijo line search procedure terminates in finite steps. 
\item The sequence $( f ( \rho_k ) )_{k \in \mathbb{N}}$ is non-increasing. 
\item For any converging sub-sequence $( \rho_k )_{k \in \mathcal{K}}$, $\mathcal{K} \subseteq \mathbb{N}$, it holds that 
\begin{equation}
\liminf \set{ D ( \rho_k ( \beta ), \rho_k ) | k \in \mathcal{K} } = 0 , \notag
\end{equation}
for every $\beta > 0$, where $\rho_k ( \beta )$ denotes the next iterate of $\rho_k$ with step size $\beta$. 
\end{enumerate}
\end{theorem}

\begin{remark}
Statement 3 is always meaningful---due to the compactness of $\mathcal{D}$, there exists at least one converging sub-sequence of $( \rho_k )_{k \in \mathbb{N}}$. 
\end{remark}

Taking limit, we obtain the following convergence guarantee. 

\begin{corollary} \label{cor_main}
If the sequence $( \rho_k )_{k \in \mathbb{N}}$ possesses a non-singular limit point, the sequence $( f ( \rho_k ) )_{k \in \mathbb{N}}$ monotonically converges to $f^\star$. 
\end{corollary}

\begin{proof}
Let $( \rho_k )_{k \in \mathcal{K}}$ be a sub-sequence converging to a non-singular matrix $\rho_\infty \in \mathcal{D}$. 
By Statement 3 of Theorem \ref{thm_main}, there exists a sub-sequence $( \rho_k )_{k \in \mathcal{K}'}$, $\mathcal{K}' \subseteq \mathcal{K}$, such that $D ( \rho_k ( \beta ), \rho_k ) \to 0$ as $k \to \infty$ in $\mathcal{K}'$. 
As $\rho_\infty$ is non-singular, we can take the limit and obtain $D ( \rho_\infty ( \beta ), \rho_\infty ) = 0$, showing that $\rho_\infty ( \beta ) = \rho_\infty$.
Lemma \ref{lem_fixed_point} in the appendix then implies that $\rho_\infty$ is a minimizer of $f$ on $\mathcal{D}$. 
Since the sequence $( f ( \rho_k ) )_{k \in \mathbb{N}}$ is non-increasing and bounded from below by $f^\star$, $\lim_{k \to \infty} f ( \rho_k )$ exists. 
We write
\begin{equation}
f^\star \leq \lim_{k \to \infty} f ( \rho_k ) = \liminf \set{ f ( \rho_k ) | k \in \mathbb{N} } \leq f ( \rho_\infty ) \leq f^\star. \iftoggle{qqed}{\tag*{\qed}}{\notag}
\end{equation}
\end{proof}

It is currently unclear to us whether convergence to the optimum holds, when there does not exist a non-singular limit point; see Section \ref{sec_convergence} for a discussion. 
One way to get around is to consider solving 
\begin{equation}
f_\lambda^\star = \min \set{ f ( \rho ) - \lambda \log \det \rho | \rho \in \mathcal{D} } , \tag{P-$\lambda$} \label{eq_problem_lambda}
\end{equation}
where $\lambda$ is a positive real number. 
As $- \log \det ( \cdot )$ is a barrier function for the set of positive semi-definite matrices \cite{Nesterov1994}, every limit point must be non-singular; otherwise, monotonicity of the sequence $( f ( \rho_k ) )_{k \in \mathbb{N}}$ (Statement 2 in Theorem \ref{thm_main}) cannot hold. 

\begin{proposition}
It holds that $\lim_{\lambda \downarrow 0} f_\lambda^\star = f^\star$. 
\end{proposition}

\begin{proof}
Notice that $- \log \det ( \cdot ) > 0$ on $\mathcal{D}$. 
We write
\begin{align}
\lim_{\lambda \downarrow 0} f_\lambda^\star = \inf_{\lambda > 0} f_\lambda^\star = \inf_{\lambda > 0} \inf_{\rho \in \mathcal{D}} f_\lambda ( \rho ) = \inf_{\rho \in \mathcal{D}} \inf_{\lambda > 0} f_\lambda ( \rho ) = \inf_{\rho \in \mathcal{D}} f ( \rho ) = f^\star , \notag
\end{align}
where $f_\lambda ( \rho ) := f ( \rho ) - \lambda \log \det \rho$. \iftoggle{qqed}{\qed}{}
\end{proof}

Existence of a non-singular limit point can be easily verified in some applications. 
For example, hedged quantum state tomography corresponds to solving \eqref{eq_problem} with the objective function 
\begin{equation}
f_{\text{HQST}} ( \rho ) := f_{\text{QST}} ( \rho ) - \lambda \log \det \rho , \notag
\end{equation}
for some $\lambda > 0$ \cite{Blume-Kohout2010}. 
As discussed above, all limit points of the iterates must be non-singular. 
Similarly in the probability simplex constraint case, if the optimization problem involves the Burg entropy as in \cite{Decarreau1992}, all limit points must be element-wisely strictly positive\footnote{For any element-wisely strictly positive vector $v := ( v_i )_{1 \leq i \leq d}$, the Burg entropy is defined as $b ( v ) := - \sum_{i = 1}^d \log v_i$.}.

\section*{Notation}
Let $g$ be a convex differentiable function. 
We denote its (effective) domain by $\dom g$, and gradient by $\nabla g$. 
If $g$ is defined on $\mathbb{R}$, we write $g'$, $g''$, and $g'''$ for its first, second, and third derivatives, respectively. 

Let $A \in \mathbb{C}^{d \times d}$. 
We denote its largest and smallest eigenvalues by $\lambda_{\max}( A )$ and $\lambda_{\min} ( A )$, respectively. 
We denote its Schatten $p$-norm by $\norm{ A }_p$.
We will only use the Hilbert-Schmidt inner product in this paper; that is, $\braket{ A, B } := \tr ( A^{\mathrm{H}} B )$ for any $A, B \in \mathbb{C}^{d \times d}$, where $A^{\mathrm{H}}$ denotes the Hermitian of $A$. 

The function $\exp( \cdot )$ and $\log ( \cdot )$ in \eqref{eq_eg} are matrix exponential and logarithm functions, respectively. 
In general, let $X \in \mathbb{C}^{d \times d}$ be Hermitian, and $X = \sum_{j} \lambda_j P_j$ be its spectral decomposition. 
Let $g$ be a real-valued function whose domain contains $\set{ \lambda_j }$. 
Then $g ( X ) := \sum_j g ( \lambda_j ) P_j$. 

Let $\rho, \sigma \in \mathcal{D}$ be non-singular. 
The negative von Neumann entropy is defined as 
\begin{equation}
h ( \rho ) := \tr ( \rho \log \rho ) - \tr ( \rho ) . \notag
\end{equation}
The quantum relative entropy is defined as 
\begin{equation}
D ( \rho, \sigma ) := \tr ( \rho \log \rho ) - \tr ( \rho \log \sigma ) - \tr ( \rho - \sigma ) , \notag
\end{equation}
which is jointly convex. 
It is easily checked that the quantum relative entropy is the Bregman divergence induced by the negative von Neumann entropy; hence, it is always non-negative. 
Pinsker's inequality says that \cite{Hiai1981}
\begin{equation}
D ( \rho, \sigma ) \geq \frac{1}{2} \norm{ \rho - \sigma }_1^2 . \notag
\end{equation}

\section{Proof of Theorem \ref{thm_main}} \label{sec_proof}

The key to our analysis is the following proposition. 

\begin{proposition} \label{prop_bertsekas}
Let $\rho \in \mathcal{D}$ be non-singular. 
Suppose that 
\begin{equation}
\Delta := \lambda_{\max} ( \nabla f ( \rho ) ) - \lambda_{\min} ( \nabla f ( \rho ) ) > 0 .  \notag 
\end{equation}
Then the mapping
\begin{equation}
\alpha \mapsto \frac{ D ( \rho ( \alpha ), \rho ) }{ \eu^{\Delta \alpha} ( \Delta \alpha - 1 ) + 1 } \label{eq_bertsekas_global}
\end{equation}
is non-increasing on $( 0, + \infty )$. 
\end{proposition}

Proposition \ref{prop_bertsekas} was inspired by a lemma due to Gafni and Bertsekas \cite{Gafni1982}, which says that the mapping
\begin{equation}
\alpha \mapsto \frac{\norm{ \Pi_{\mathcal{D}} ( \rho - \alpha \nabla f ( \rho ) ) - \rho }_{\text{F}}}{\alpha} \label{eq_bertsekas_original}
\end{equation}
is non-increasing on $[ 0, + \infty )$, where $\Pi_{\mathcal{D}}$ denotes projection onto $\mathcal{D}$ with respect to the Frobenius norm $\norm{ \cdot }_{\text{F}}$. 
The lemma of Gafni and Bertsekas was proved by an Euclidean geometric argument; see \cite{Bertsekas2016} for an illustration. 
In comparison, we will prove Proposition \ref{prop_bertsekas} by exploiting the self-concordant likeness of the log-partition function. 


We prove Proposition \ref{prop_bertsekas} in Section \ref{sec_bertsekas}. 
Then we prove the three statements in Theorem \ref{thm_main} separately in the following three sub-sections. 
To simplify the presentation, we put some necessary technical lemmas in Appendix \ref{app_technical}. 

\subsection{Self-concordant Likeness of the Log-Partition Function and \\ Proof of Proposition \ref{prop_bertsekas}} \label{sec_bertsekas}

For any non-singular $\rho \in \mathcal{D}$ and $\alpha > 0$, define
\begin{equation}
\varphi ( \alpha; \rho ) := \log \tr \exp \left[ \log ( \rho ) - \alpha \nabla f ( \rho ) \right] , \notag
\end{equation}
which, in statistical physics, is the log-partition function of the Gibbs state for the Hamiltonian $H_\alpha := - \log ( \rho ) + \alpha \nabla f ( \rho )$ at temperature $1$.
We will simply write $\varphi ( \alpha )$ instead of $\varphi ( \alpha; \rho )$, when the corresponding $\rho$ is clear from the context or irrelevant. 

The log-partition function is indeed closely related to the EG method, as shown by the following lemma. 

\begin{lemma} \label{lem_divergence}
For any non-singular $\rho \in \mathcal{D}$ and $\alpha > 0$, it holds that
\begin{equation}
D ( \rho ( \alpha ), \rho ) = \varphi ( 0 ) - \left[ \varphi ( \alpha ) - \varphi' ( \alpha ) ( 0 - \alpha ) \right] . \notag
\end{equation}
\end{lemma}

\begin{proof}
Direct calculation. \iftoggle{qqed}{\qed}{}
\end{proof}

We say that a three times continuously differentiable convex function $g$ is \emph{$\mu$-self-concordant like}, if $\abs{ g''' ( x ) } \leq \mu g'' ( x )$ for all $x$ \cite{Bach2010,Bach2014,Tran-Dinh2015}. 

\begin{lemma} \label{lem_self-concordance}
For any non-singular $\rho \in \mathcal{D}$, the function $\varphi ( \alpha )$ is $\Delta$-self-concordant like, where $\Delta := \lambda_{\max} ( \nabla f ( \rho ) ) - \lambda_{\min} ( \nabla f ( \rho ) )$. 
\end{lemma}

\begin{proof}
Lemma \ref{lem_moments} shows that 
\begin{equation}
\varphi'' ( \alpha ) = \mathsf{E} \, \left( \eta_\alpha - \mathsf{E}\, \eta_\alpha \right)^2, \quad \varphi''' ( \alpha ) = \mathsf{E} \, \left( \eta_\alpha - \mathsf{E}\, \eta_\alpha \right)^3 , \notag
\end{equation}
where $\eta_\alpha$ is a random variable taking values in $[ - \lambda_{\max} ( \nabla f ( \rho ) ), - \lambda_{\min} ( \nabla f ( \rho ) ) ]$. 
The lemma follows. \iftoggle{qqed}{\qed}{}
\end{proof}

The following sandwich inequality follows from self-concordant likeness \cite{Tran-Dinh2015}. 

\begin{lemma} \label{lem_self-concordance_bound}
For any non-singular $\rho \in \mathcal{D}$, it holds that
\begin{align}
\frac{ \left( \eu^{- \Delta \alpha} + \Delta \alpha - 1 \right) }{ \Delta^2 } \varphi'' ( \alpha ) & \leq \varphi ( 0 ) - \left[ \varphi ( \alpha ) - \varphi' ( \alpha ) ( 0 - \alpha ) \right] \notag \\
& \leq \frac{ \left( \eu^{\Delta \alpha} - \Delta \alpha - 1 \right) }{ \Delta^2 } \varphi'' ( \alpha ) . \notag
\end{align}
\end{lemma}

\begin{remark} \label{rem_peierls_bogoliubov}
The lower bound improves upon the Peierls-Bogoliubov inequality \cite{Ohya1993}, which says that
\begin{equation}
0 \leq \varphi ( 0 ) - \left[ \varphi ( \alpha ) - \varphi' ( \alpha ) ( 0 - \alpha ) \right] . \notag
\end{equation}
Notice that lower bound provided by Lemma \ref{lem_self-concordance_bound} is always non-negative.
\end{remark}

Now we are ready to prove Proposition \ref{prop_bertsekas}. 

\begin{proof}[Proposition \ref{prop_bertsekas}]
We look for a differentiable function $\chi: ( 0, + \infty ) \to ( 0, + \infty )$, such that the mapping
\begin{equation}
g ( \alpha ) := \frac{D ( \rho ( \alpha ), rho )}{\chi ( \alpha )} \notag
\end{equation}
is non-increasing on $( 0, + \infty )$. 
Note that $g$ is non-increasing if and only if $g' \leq 0$ on $( 0, + \infty )$. 
Applying Lemma \ref{lem_divergence}, a direct calculation gives
\begin{equation}
g' ( \alpha ) = \frac{\alpha \varphi'' ( \alpha ) \chi ( \alpha ) - \left\{ \varphi ( 0 ) - \left[ \varphi ( \alpha ) + \varphi' ( \alpha ) ( 0 - \alpha ) \right] \right\} \chi' ( \alpha )}{\left[ \chi' ( \alpha ) \right]^2} . \notag
\end{equation}
Therefore, $g' ( \alpha ) \leq 0$ if and only if the numerator is negative, i.e., 
\begin{equation}
( \log \chi )' ( \alpha ) \geq \frac{\alpha \varphi'' ( \alpha )}{\varphi ( 0 ) - \left[ \varphi ( \alpha ) + \varphi' ( \alpha ) ( 0 - \alpha ) \right]} , \notag
\end{equation}
where we have used the fact that $\chi' / \chi = ( \log \chi )'$. 
By Lemma \ref{lem_self-concordance_bound}, we can set 
\begin{equation}
( \log \chi )' ( \alpha ) = \frac{\Delta^2 \alpha}{\eu^{- \Delta \alpha} + \Delta \alpha - 1} .  \notag
\end{equation}
Solving the equation gives $\chi ( \alpha ) := \eu^{\Delta \alpha} ( \Delta \alpha - 1 ) + 1$. \iftoggle{qqed}{\qed}{}
%
\end{proof}

For convenience, we will apply Proposition \ref{prop_bertsekas} via the following corollary. 

\begin{corollary} \label{cor_bertsekas}
Let $\rho \in \mathcal{D}$ be non-singular and $\bar{\alpha} > 0$. 
Suppose that $\Delta := \lambda_{\max} ( \nabla f ( \rho ) ) - \lambda_{\min} ( \nabla f ( \rho ) )$ is strictly positive.
It holds that 
\begin{equation}
\frac{ D ( \rho ( \alpha ), \rho ) }{ \alpha^2 } \geq \kappa D ( \rho ( \bar{\alpha} ), \rho ) , \quad \forall \alpha \in ( 0, \bar{\alpha} ] , \notag
\end{equation}
where $\kappa := \left\{ 2 \left[ \eu^{\Delta \bar{\alpha}} ( \Delta \bar{\alpha} - 1 ) + 1 \right] \right\}^{-1} \Delta^2$. 
\end{corollary}

\begin{proof}
Define $g ( \alpha ) := \eu^{\Delta \alpha} ( \Delta \alpha - 1 ) + 1 - ( \Delta^2 / 2 ) \alpha^2$. 
Then $g ( 0 ) = 0$, and 
\begin{equation}
g' ( \alpha ) = \alpha \left[ \eu^{\Delta \alpha} \Delta^2 - \Delta^2 \right] \geq \alpha ( \Delta^2 - \Delta^2 ) = 0, \quad \forall \alpha \in ( 0, + \infty ) . \notag
\end{equation}
Therefore, $g ( \alpha ) \geq 0$ on $( 0, + \infty )$, i.e., 
\begin{equation}
\eu^{\Delta \alpha} ( \Delta \alpha - 1 ) + 1 \geq \frac{\Delta^2}{2} \alpha^2 , \quad \forall \alpha \in ( 0, + \infty ) . \notag 
\end{equation}
By Proposition \ref{prop_bertsekas}, we write
\begin{equation}
\frac{D ( \rho ( \alpha ), \rho )}{\frac{\Delta^2}{2} \alpha^2} \geq \frac{D ( \rho ( \alpha ), \rho )}{ \eu^{\Delta \alpha} ( \Delta \alpha - 1 ) + 1 } \geq \frac{D ( \rho ( \bar{ \alpha } ), \rho )}{ \eu^{\Delta \bar{\alpha}} ( \Delta \bar{\alpha} - 1 ) + 1 }, \quad \forall \alpha \in ( 0, \bar{\alpha} ] . \iftoggle{qqed}{\tag*{\qed}}{\notag}
\end{equation}
\end{proof}

\subsection{Proof of Statement 1}

The first statement is a direct consequence of the following proposition. 

\begin{proposition} \label{prop_armijo_terminates}
For every non-singular $\rho \in \mathcal{D}$, there exists some $\alpha_\rho > 0$ such that 
\begin{equation}
f ( \rho ( \alpha ) ) \leq f ( \rho ) + \tau \braket{ \nabla f ( \rho ), \rho ( \alpha ) - \rho } , \quad \forall \alpha \in [ 0, \alpha_\rho ] . \label{eq_armijo_to_be_checked}
\end{equation}
\end{proposition}

Recall that $\tau$ is the parameter in Armijo line search. 

\begin{proof}
If $\rho$ is a minimizer, by Lemma \ref{lem_fixed_point}, we have $\rho ( \alpha ) = \rho$ for all $\alpha \in [ 0, + \infty )$, and the proposition follows. 
Suppose that $\rho$ is not a minimizer in the rest of this proof.
By Lemma \ref{lem_fixed_point}, we have $D ( \rho ( \alpha ), \rho ) > 0$ for all $\alpha \in ( 0, + \infty )$. 
By the mean-value theorem, we write
\begin{equation}
f ( \rho ( \alpha ) ) - f ( \rho ) = \braket{ \nabla f ( \sigma ), \rho ( \alpha ) - \rho } ,  \notag 
\end{equation}
for some $\sigma$ in the line segment joining $\rho ( \alpha )$ and $\rho$. 
Then \eqref{eq_armijo_to_be_checked} can be equivalently written as
\begin{equation}
\braket{ \nabla f ( \sigma ) - \nabla f ( \rho ), \rho ( \alpha ) - \rho } \leq - ( 1 - \tau ) \braket{ \nabla f ( \rho ), \rho ( \alpha ) - \rho }, \quad \forall \alpha \in [ 0, \alpha_\rho ] . \label{eq_armijo_to_be_checked_ver2}
\end{equation}
By Lemma \ref{lem_inner_product}, \eqref{eq_armijo_to_be_checked_ver2} holds if
\begin{equation}
\braket{ \nabla f ( \sigma ) - \nabla f ( \rho ), \rho ( \alpha ) - \rho } \leq \frac{ ( 1 - \tau ) D ( \rho ( \alpha ), \rho ) }{ \alpha } , \quad \forall \alpha \in [ 0, \alpha_\rho ] . \label{eq_armijo_to_be_checked_ver3}
\end{equation}
Consider two cases. 
\begin{itemize}
\item If $\lambda_{\max} ( \nabla f ( \rho ) ) = \lambda_{\min} ( \nabla f ( \rho ) )$, then $\nabla f ( \rho )$ is a multiple of the identity. 
We have 
\begin{equation}
\braket{ \nabla f ( \rho ), \sigma - \rho } = 0 , \quad \forall \sigma \in \mathcal{D} ; \notag
\end{equation}
showing that $\rho$ is a minimizer. 
By Lemma \ref{lem_fixed_point}, the proposition follows for every $\alpha_\rho > 0$. 

\item Otherwise, set $\alpha_\rho \leq \bar{\alpha}$. 
By Corollary \ref{cor_bertsekas}, there exists some $\kappa > 0$, such that 
\begin{equation}
\frac{D ( \rho ( \alpha ), \rho )}{\alpha} \geq \sqrt{ D ( \rho ( \alpha ), \rho ) } \sqrt{ \kappa D ( \rho ( \bar{\alpha} ), \rho ) }, \quad \forall \alpha \in [ 0, \alpha_\rho ] . \notag
\end{equation}
Applying H\"{o}lder's inequality and Pinsker's inequality, we write
\begin{align}
\braket{ \nabla f ( \sigma ) - \nabla f ( \rho ), \rho ( \alpha ) - \rho } &\leq \norm{ \nabla f ( \sigma ) - \nabla f ( \rho ) }_\infty \norm{ \rho ( \alpha ) - \rho }_1 \notag \\
& \leq \norm{ \nabla f ( \sigma ) - \nabla f ( \rho ) }_\infty \sqrt{ 2 D ( \rho ( \alpha ), \rho ) } . \notag
\end{align}
Then \eqref{eq_armijo_to_be_checked_ver3} holds if
\begin{equation}
\norm{ \nabla f ( \sigma ) - \nabla f ( \rho ) }_\infty \sqrt{ 2 } \leq ( 1 - \tau ) \sqrt{ \kappa D ( \rho ( \bar{\alpha} ), \rho ) } , \quad \forall \alpha \in [ 0, \alpha_\rho ] \notag
\end{equation}
Recall that a convex differentiable function is continuously differentaible \cite{Rockafellar1970b}. 
Notice that $\rho ( \alpha )$ is continuous in $\alpha$. 
As the right-hand side is a strictly positive constant by Lemma \ref{lem_fixed_point}, the proposition follows for a small enough $\alpha_\rho$. \iftoggle{qqed}{\qed}{}
\end{itemize}
\end{proof}

\subsection{Proof of Statement 2}

By the definition of Armijo line search and Lemma \ref{lem_inner_product}, we have
\begin{equation}
f ( \rho_k ) \leq f ( \rho_{k - 1} ) + \tau \braket{ \nabla f ( \rho_{k - 1} ), \rho_k - \rho_{k - 1} } \leq f ( \rho_{k - 1} ) - \frac{\tau D ( \rho_k, \rho_{k - 1} )}{\alpha_k} . \notag
\end{equation}
As the quantum relative entropy $D$ is always non-negative, it follows that the sequence $( f ( \rho_k ) )_{k \in \mathbb{N}}$ is non-increasing. 

\subsection{Proof of Statement 3}

If $\rho_k$ is a minimizer for some $k \in \mathbb{N}$, by Lemma \ref{lem_fixed_point}, it holds that $\rho_{k'} = \rho_k$ for all $k' > k$, and the statement trivially follows. 
In the rest of this sub-section, we assume that $\rho_k$ is not a minimizer for all $k$; then by Lemma \ref{lem_fixed_point}, it holds that $\rho_k \neq \rho_{k - 1}$ for all $k \in \mathbb{N}$. 

Let $( \rho_k )_{k \in \mathcal{K}}$ be a sub-sequence converging to a limit point $\rho_\infty \in \mathcal{D}$, which exists due to the compactness of $\mathcal{D}$. 
Then $\rho_\infty$ must lie in $\inte \dom f$; otherwise, monotonicity of the sequence $( f ( \rho_k ) )_{k \in \mathbb{N}}$ (Statement 2 of Theorem \ref{thm_main}) cannot hold. 
As $f$ is continuously differentiable, it holds that 
\begin{align}
\frac{ \Delta_\infty }{ 2 } \leq \lambda_{\max} ( \nabla f ( \rho_k ) ) - \lambda_{\min} ( \nabla f ( \rho_k ) ) \leq 2 \Delta_\infty , \label{eq_gradient_limit}
\end{align}
for large enough $k \in \mathcal{K}$, where $\Delta_\infty := \lambda_{\max} ( \nabla f ( \rho_\infty ) ) - \lambda_{\min} ( \nabla f ( \rho_\infty )$. 

\begin{lemma} \label{lem_zero_delta}
If $\Delta_\infty = 0$, then $\liminf \set{ D ( \rho_k ( \beta ), \rho_k ) | k \in \mathcal{K} } = 0$ for every $\beta \in [ 0, + \infty )$. 
\end{lemma}

\begin{proof}
Define $\Delta_k := \lambda_{\max} ( \nabla f ( \rho_k ) ) - \lambda_{\min} ( \nabla f ( \rho_k ) )$; then $\Delta_k \to \Delta_\infty = 0$. 
Define $\varphi_k : \alpha \mapsto \varphi ( \alpha; \rho_k )$. 
By Lemma \ref{lem_self-concordance_bound} and Corollary \ref{cor_bounded_hessian}, we have
\begin{align}
\varphi_k ( 0 ) - \left[ \varphi_k ( \beta ) - \varphi_k' ( \beta ) ( 0 - \beta ) \right] & \leq \frac{ \left( \eu^{\Delta_k \beta} - \Delta_k \beta - 1 \right) }{ \Delta_k^2 } \varphi_k'' ( \beta ) \notag \\
& \leq \frac{ \left( \eu^{\Delta_k \beta} - \Delta_k \beta - 1 \right) }{ 4 }. \notag
\end{align}
By Lemma \ref{lem_divergence}, we obtain
\begin{align}
0 & \leq \liminf \set{ D ( \rho_k ( \beta ), \rho_k ) | k \in \mathcal{K} } \notag \\
& = \liminf \set{ \varphi_k ( 0 ) - \left[ \varphi_k ( \beta ) - \varphi_k' ( \beta ) ( 0 - \beta ) \right] | k \in \mathcal{K} } \notag \\
& \leq \frac{\eu^{0} - 0 - 1}{4} = 0 . \iftoggle{qqed}{\tag*{\qed}}{\notag}
\end{align}
\end{proof}

Suppose that $\Delta_\infty > 0$. 
We have the following analogy of Corollary \ref{cor_bertsekas} for large enough $k \in \mathcal{K}$: 

\begin{corollary} \label{cor_bertsekas_asymptotic}
Suppose that $\Delta_\infty > 0$ and $\rho_k$ is not a minimizer for every $k \in \mathcal{K}$. 
There exists some $\kappa > 0$, such that 
\begin{equation}
\frac{D ( \rho_k ( \alpha ), \rho_k )}{\alpha^2} \geq \kappa D ( \rho_k ( \bar{\alpha} ), \rho_k ) , \quad \forall \alpha \in ( 0, \bar{\alpha} ] , \notag
\end{equation}
for large enough $k \in \mathcal{K}$. 
\end{corollary}

\begin{proof}
Recall that \eqref{eq_gradient_limit} provides both upper and lower bounds of $\lambda_{\max} ( \nabla f ( \rho_k ) ) - \lambda_{\min} ( \nabla f ( \rho_k ) )$, for large enough $k \in \mathcal{K}$. 
With regard to Corollary \ref{cor_bertsekas}, it suffices to set
\begin{equation}
\kappa = \frac{\Delta_\infty^2}{4 \left[ \eu^{2 \Delta_\infty \bar{\alpha}} ( 2 \Delta_\infty \bar{\alpha} - 1 ) + 1 \right]} . \iftoggle{qqed}{\tag*{\qed}}{\notag}
\end{equation}
\end{proof}

Based on Corollary \ref{cor_bertsekas_asymptotic}, we prove the following proposition. 

\begin{proposition} \label{prop_convergence_suppl}
Suppose that $\Delta_\infty > 0$ and $\rho_k$ is not a minimizer for every $k \in \mathcal{K}$. 
It holds that $\liminf \set{ D ( \rho_k ( \bar{\alpha} ), \rho_k ) | k \in \mathcal{K} } = 0$. 
\end{proposition}

The proof of Proposition \ref{prop_convergence_suppl} can be found in Section \ref{sec_convergence_suppl}, which essentially follows the strategy of Gafni and Bertsekas \cite{Gafni1982} with necessary modifications. 

To summarize, we have proved that for any converging sub-sequence $( \rho_k )_{k \in \mathcal{K}}$, there exists some $\gamma > 0$ such that
\begin{equation}
\liminf \set{ D ( \rho_k ( \gamma ), \rho_k ) | k \in \mathcal{K} } = 0 . \notag 
\end{equation}
For the case where $\rho_k$ is a minimizer for some $k \in \mathcal{K}$ or $\Delta_\infty = 0$, $\gamma$ can be any strictly positive real number. 
Otherwise, we set $\gamma = \bar{\alpha}$ by Proposition \ref{prop_convergence_suppl}. 

By Lemma \ref{lem_divergence} and Lemma \ref{lem_self-concordance_bound}, it holds that
\begin{align}
0 & \leq \liminf \Set{ \frac{ \left( \eu^{- ( 1 / 2 ) \Delta_\infty \gamma} + ( 1 / 2 ) \Delta_\infty \gamma - 1 \right) }{ \gamma^2 } \varphi_k'' ( \gamma ) | k \in \mathcal{K} } \notag \\
& \leq \liminf \set{ D ( \rho_k ( \gamma ), \rho_k ) | k \in \mathcal{K} } = 0 , \notag
\end{align}
showing that $\liminf \set{ \varphi_k'' ( \gamma ) | k \in \mathcal{K} } = 0$. 
Applying Lemma \ref{lem_divergence} and Lemma \ref{lem_self-concordance_bound} again, we obtain
\begin{align}
0 & \leq \liminf \set{ D ( \rho_k ( \beta ), \rho_k ) | k \in \mathcal{K} | k \in \mathcal{K} } \notag \\
& \leq \liminf \Set{ \frac{ \left( \eu^{ 2 \Delta_\infty \beta} - 2 \Delta_\infty \beta - 1 \right) }{ \beta^2 } \varphi_k'' ( \beta ) | k \in \mathcal{K} } = 0 , \notag
\end{align}
for any $\beta \in ( 0, + \infty )$.
This proves Statement 3 of Theorem \ref{thm_main}. 

\section{Concluding Remarks}

Assuming only differentiability of the objective function, we have proved that the EG method with Armijo line search monotonically converges to the optimum, if the sequence of iterates possesses a non-singular limit point. 
Our proof exploits the self-concordant likeness of the log-partition function, which is of independent interest; 
in particular, Lemma \ref{lem_self-concordance_bound} improves upon the Peierls-Bogoliubov inequality. 

\subsection{Importance of Self-Concordant Likeness}

With regard to \eqref{eq_bertsekas_original}, one may suspect whether it suffices, for the convergence analysis, to prove the following: 
There exists some $\epsilon > 0$, such that the mapping $\alpha \mapsto \alpha^{- \epsilon} D ( \rho ( \alpha ), \rho )$ is non-increasing on $( 0, \bar{\alpha} ]$ for every non-singular $\rho \in \mathcal{D}$. 
Indeed, following the proof strategy for Proposition \ref{prop_bertsekas}, we obtain the following result \emph{without self-concordant likeness}. 

\begin{proposition}
Let $\rho \in \mathcal{D}$ be non-singular. 
Define 
\begin{equation}
M := \sup \set{ \varphi'' ( \alpha; \rho ) | \alpha \in ( 0, \bar{\alpha} ) } , \quad m := \inf \set{ \varphi'' ( \alpha; \rho ) | \alpha \in ( 0, \bar{\alpha} ) } . \notag
\end{equation}
Suppose that $m > 0$. 
Then the mapping $\alpha \mapsto \alpha^{ - \epsilon} D ( \rho ( \alpha ), \rho )$ is non-increasing on $( 0, \bar{\alpha} )$, where $\epsilon := 2 M / m$. 
\end{proposition}

\begin{remark}
For the case where $m = 0$, Lemma \ref{lem_moments} implies that $\nabla f$ must be a multiple of the identity. 
Then it is easily checked that $\rho$ is a minimizer as it verifies the optimality condition. 
\end{remark}

Then in the proof of Proposition \ref{prop_armijo_terminates}, for example, the condition we need to verify becomes: 
\begin{equation}
\norm{ \nabla f ( \sigma ) - \nabla f ( \rho ) }_\infty \sqrt{ 2 } \leq ( 1 - \tau ) \alpha^{\epsilon / 2 - 1} \sqrt{ \frac{ D ( \rho ( \bar{\alpha} ), \rho ) }{ \bar{\alpha}^2 } } , \quad \forall \alpha \in [ 0, \alpha_\rho ] . \notag
\end{equation}
Notice that $\epsilon \geq 2$ by definition. 
Both sides can converge to zero as $\alpha \to 0$, so in general, there does not exist a small enough $\alpha_\rho$ that verifies the condition. 
Moreover, because $\alpha^\epsilon \leq \alpha^2$ for $\alpha \in [ 0, 1 ]$, it is impossible to obtain an analogue of Corollary \ref{cor_bertsekas}. 

The point in our analysis is to show that there exists some $\chi ( \alpha )$, bounded from below by $\alpha^2$ for every $\alpha$ close to zero, such that the mapping $\alpha \mapsto D ( \rho ( \alpha ), \rho ) / \chi ( \alpha )$ is non-increasing. 
This is where self-concordant likeness of the log-partition function comes into play. 

\subsection{Extensions for the Probability Simplex and Spectrahedron Constraints} \label{sec_extensions}

The EG method can be extended for the spectrahedron and probability simplex constraints; in fact, the EG method is arguably better known for these two cases \cite{Auslender2006,Beck2003,Kivinen1997,Tsuda2005}. 
For the former case, the iteration rule writes exactly the same as \eqref{eq_eg}, and is equivalent to \eqref{eq_eg_equivalent} with $\mathcal{D}$ replaced by the spectrahedron $\mathcal{S}$. 
For the latter case, the iteration rule becomes element-wise (see, e.g., \cite{Beck2003}) and is equivalent to \eqref{eq_eg_equivalent}, with $\mathcal{D}$ replaced by the probability simplex $\mathcal{P}$, and the quantum relative entropy replaced by the (classical) relative entropy. 
The Armijo line search rule applies without modification. 

It is easily checked that our proof holds without modification for the spectrahedron constraint. 
As a vector in $\mathbb{R}^d$ is equivalent to a diagonal matrix in $\mathbb{R}^{d \times d}$, it can be easily checked that the statements in Theorem \ref{thm_main} applies to the probability simplex constraint. 
Corollary \ref{cor_main} also holds true for these two constraints with slight modification---for the probability simplex constraint, non-singularity should be replaced by element-wise strict positivity. 

\subsection{Convergence with Possibly Singular Limit Points} \label{sec_convergence}

Corollary \ref{cor_main} requires existence of at least one non-singular limit point. 
Can this condition be removed? 

Suppose that the sequence $( \rho_k )_{k \in \mathbb{N}}$ has a possibly singular limit point $\rho_\infty$, around which $\nabla f$ is locally $L$-Lipschitz continuous with respect to the Schatten $1$-norm. 
Let $( \rho_k )_{k \in \mathcal{K}}$, $\mathcal{K} \subset \mathbb{N}$, be a sub-sequence converging to $\rho_\infty$. 
Then following the proof of the second part of Proposition \ref{prop_convergence_suppl}, it is easily checked that $\liminf \set{ \alpha_k | k \in \mathcal{K} } = 0$ implies
\begin{equation}
\alpha_k \geq \frac{L}{r ( 1 - \tau )} , \notag
\end{equation}
a contradiction; hence, $\liminf \set{ \alpha_k | k \in \mathcal{K} }$ must be strictly positive. 
Then following the proof in \cite{Auslender2006}, it holds that the sequence $( f ( \rho_k ) )_{k \in \mathbb{N}}$ monotonically converges to the optimal value. 

In general without the local Lipschitz gradient condition, we conjecture that convergence to the optimum cannot be guaranteed.
However, we have not found a counter-example. 

\section*{Acknowledgements}

We thank Ya-Ping Hsieh for his comments. 
This work was supported by SNF 200021-146750 and ERC project time-data 725594.

\appendix

\section{Proof of Proposition \ref{prop_qst_is_hard}} \label{sec_qst_is_hard}

Consider the two-dimensional case, where $\rho = ( \rho_{i, j} )_{1 \leq i, j \leq 2} \in \mathbb{C}^{2 \times 2}$. 
Define $e_1 := ( 1, 0 )$ and $e_2 := ( 0, 1 )$. 
Suppose that there are only two summands, with $M_1 = e_1 \otimes e_1$ and $M_2 = e_2 \otimes e_2$. 
Then we have $f ( \rho ) = - \log \rho_{1,1} - \log \rho_{2,2}$. 
It suffices to disprove all properties for this specific $f$ on the set of diagonal density matrices.
Hence, we will focus on the function $g ( x, y ) := - \log x - \log y$, defined for any $x, y > 0$ such that $x + y = 1$.

As either $x$ or $y$ can be arbitrarily close to zero, $g$ cannot be Lipschitz continuous in itself or its gradient due to the logarithmic functions.
Define the entropy function 
\begin{equation}
h(x, y) := - x \log x - y \log y + x + y, \notag
\end{equation}
with the convention $0 \log 0 = 0$.
Then $g$ is $L$-smooth relative to the relative entropy, if and only if $- L h - g$ is convex.
It suffices to check the positive semi-definiteness of the Hessian of $- L h - g$. 
A necessary condition for the Hessian to be positive semi-definite is that 
\begin{equation}
- L \frac{\partial^2 h}{\partial x^2} ( x, y ) - \frac{\partial^2 g}{\partial x^2} ( x, y ) = \frac{L}{x} - \frac{1}{x^2} \geq 0 , \notag
\end{equation}
for all $x \in ( 0, 1 )$, which cannot hold for $x < ( 1 / L )$, for any fixed $L > 0$.  \iftoggle{qqed}{\qed}{}

\section{Technical Lemmas Necessary for Section \ref{sec_proof}} \label{app_technical}

Recall the definition: 
\begin{equation}
\rho ( \alpha ) := C_\rho^{-1} \exp \left[ \log ( \rho ) - \alpha \nabla f ( \rho ) \right] , \notag
\end{equation}
for every non-singular $\rho \in \mathcal{D}$ and $\alpha \geq 0$, where $C_\rho$ is the positive real number normalizing the trace of $\rho ( \alpha )$. 

\begin{lemma} \label{lem_inner_product}
For every non-singular $\rho \in \mathcal{D}$ and $\alpha > 0$, it holds that
\begin{equation}
\braket{ \nabla f ( \rho ), \rho ( \alpha ) - \rho } \leq - \frac{D ( \rho ( \alpha ), \rho )}{ \alpha } . \notag
\end{equation}
\end{lemma}

\begin{proof}
The equivalent formulation of the EG method, \eqref{eq_eg_equivalent}, implies that
\begin{equation}
\alpha \braket{ \nabla f ( \rho ), \rho ( \alpha ) - \rho } + D ( \rho ( \alpha ), \rho ) \leq \alpha \braket{ \nabla f ( \rho ), \rho - \rho } + D ( \rho, \rho ) = 0 . \iftoggle{qqed}{\tag*{\qed}}{\notag}
\end{equation}
\end{proof}

\begin{lemma} \label{lem_fixed_point}
Let $\rho \in \mathcal{D}$ be non-singular. 
If $\rho$ is a minimizer of $f$ on $\mathcal{D}$, then $\rho ( \alpha ) = \rho$ for all $\alpha \geq 0$. 
If $\rho ( \alpha ) = \rho$ for some $\alpha > 0$, then $\rho$ is a minimizer of $f$ on $\mathcal{D}$. 
\end{lemma}

\begin{proof}
The optimality condition says that $\rho \in \inte \mathcal{D}$ is a minimizer, if and only if
\begin{equation}
\braket{ \nabla f ( \rho ), \sigma - \rho } \geq 0 , \quad \forall \sigma \in \mathcal{D} . \notag
\end{equation}
For any $\alpha > 0$, we can equivalently write
\begin{equation}
\braket{ \alpha \nabla f ( \rho ) + \left[ \nabla h ( \rho ) - \nabla h ( \rho ) \right], \sigma - \rho } \geq 0 , \quad \forall \sigma \in \mathcal{D} , \label{eq_redundant_optimality_conditoin}
\end{equation}
where $h$ denotes the negative von Neumann entropy function, i.e., 
\begin{equation}
h ( \rho ) := \tr ( \rho \log \rho ) - \tr \rho . \notag
\end{equation}
Notice that the quantum relative entropy $D$ is the Bregman divergence induced by the negative von Neumann entropy. 
It is easily checked, again by the optimality condition, that \eqref{eq_redundant_optimality_conditoin} is equivalent to 
\begin{equation}
\rho = \argmin \set{ \alpha \braket{ \nabla f ( \rho ), \sigma - \rho } + D ( \sigma, \rho ) | \sigma \in \mathcal{D} } = \rho ( \alpha ) . \iftoggle{qqed}{\tag*{\qed}}{\notag}
\end{equation}
\end{proof}

For every non-singular $\rho \in \mathcal{D}$ and $\alpha \geq 0$, define
\begin{equation}
G := - \nabla f ( \rho ), \quad H_\alpha := \log \rho + \alpha G. \notag
\end{equation}
Let $G = \sum_j \lambda_j P_j$ be the spectral decomposition of $G$. 
Define $\eta_\alpha$ as a random variable satisfying 
\begin{equation}
\mathsf{P} \left( \eta_\alpha = \lambda_j \right) = \frac{ \tr \left( P_j \exp ( H_\alpha ) \right) }{ \tr \exp ( H_\alpha ) } ; \notag
\end{equation}
it is easily checked that $\mathsf{P} \left( \eta_\alpha = \lambda_j \right) > 0$ for all $j$, and the probabilities sum to one. 

\begin{lemma} \label{lem_moments}
For any $\alpha \in \mathbb{R}$, it holds that
\begin{align}
\varphi' ( \alpha ) = \mathsf{E}\, \eta_\alpha, \quad \varphi'' ( \alpha ) = \mathsf{E} \left( \eta_\alpha - \mathsf{E}\, \eta_\alpha \right)^2, \quad \varphi''' ( \alpha ) = \mathsf{E} \left( \eta_\alpha - \mathsf{E}\, \eta_\alpha \right)^3 . \notag
\end{align}
\end{lemma}

\begin{proof}
Notice that 
\begin{equation}
\mathsf{E}\, \eta_\alpha^n = \frac{\tr ( G^n \exp ( H_\alpha ) )}{\tr \exp ( H_\alpha )} , \notag
\end{equation}
for any $n \in \mathbb{N}$. 
Define $\sigma_\alpha := \exp ( H_\alpha ) / \tr \exp ( H_\alpha )$. 
A direct calculation gives
\begin{align}
\varphi' ( \alpha ) & = \tr ( G \sigma_\alpha ) , \quad \varphi'' ( \alpha ) = \tr ( G^2 \sigma_\alpha ) - \left( \tr ( G \sigma_\alpha ) \right)^2 , \notag \\
\varphi''' ( \alpha ) & = \tr ( G^3 \sigma_\alpha ) - 3 \tr ( G^2 \sigma_\alpha ) \tr ( G \sigma_\alpha ) + 2 \left( \tr ( G \sigma_\alpha ) \right)^3 . \notag
\end{align}
The lemma follows. \iftoggle{qqed}{\qed}{}
\end{proof}

Since $\eta_\alpha$ is a bounded random variable, it follows that $\varphi''$ is bounded from above. 

\begin{corollary} \label{cor_bounded_hessian}
It holds that $\varphi'' ( \alpha ) \leq ( 1 / 4 ) \Delta^2$, where $\Delta := \lambda_{\max} ( \nabla f ( \rho ) ) - \lambda_{\min} ( \nabla f ( \rho ) )$. 
\end{corollary}

\begin{proof}
Recall that the variance of a random variable taking values in $[a, b]$ is bounded from above by $( b - a )^2 / 4$. \iftoggle{qqed}{\qed}{}
\end{proof}

\section{Proof of Proposition \ref{prop_convergence_suppl}} \label{sec_convergence_suppl}

Suppose that $\underline{\alpha} := \liminf \set{ \alpha_k | k \in \mathcal{K} } > 0$. 
By the Armijo line search rule and Corollary \ref{cor_bertsekas}, we write
\begin{align}
f ( \rho_k ) - f ( \rho_{k + 1} ) & \geq - \tau \braket{ \nabla f ( \rho_{k} ), f ( \rho_{k + 1} ) - f ( \rho_{k} ) } \notag \\
& \geq \tau \alpha_k^{-1} D ( \rho_{k + 1}, \rho_k ) \notag \\
& = \tau \alpha_k \alpha_k^{-2} D ( \rho_k ( \alpha_k ), \rho_k ) \notag \\
& \geq \tau \underline{\alpha} \kappa D ( \rho_k ( \bar{\alpha} ), \rho_k ) \notag \\
& \geq 0 , \notag
\end{align}
for large enough $k \in \mathcal{K}$. 
Taking limit, we obtain that $D ( \rho_k ( \bar{\alpha} ), \rho_k ) \to 0$ as $k \to \infty$ in $\mathcal{K}$. 

Suppose that $\liminf \set{ \alpha_k | k \in \mathcal{K} } = 0$. 
Let $( \alpha_k )_{k \in \mathcal{K}'}$, $\mathcal{K}' \subseteq \mathcal{K}$, be a sub-sequence converging to zero. 
According to the Armijo rule, we have
\begin{equation}
f ( \rho_k ( r^{-1} \alpha_k ) ) - f ( \rho_k ) > \tau \braket{ \nabla f ( \rho_k ), \rho ( r^{-1} \alpha_k ) - \rho ( \alpha_k ) } , \label{eq_blabla}
\end{equation}
for large enough $k \in \mathcal{K}$. 
The mean value theorem says that the left-hand side equals $\braket{ \nabla f ( \sigma ), \rho_k ( r^{-1} \alpha_k ) - \rho_k }$ for some $\sigma$ in the line segment jointing $\rho_k ( r^{-1} \alpha_k )$ and $\rho_k$. 
Then \eqref{eq_blabla} can be equivalently written as
\begin{equation}
\braket{ \nabla f ( \sigma ) - \nabla f ( \rho_k ), \rho_k ( r^{-1} \alpha_k ) - \rho_k } > - ( 1 - \tau ) \braket{ \nabla f ( \rho_k ), \rho_k ( r^{-1} \alpha_k ) - \rho_k ( \alpha_k ) } . \label{eq_blabla_2}
\end{equation}
By Pinsker's inequality and H\"{o}lder's inequality, we obtain
\begin{align}
\norm{ \nabla f ( \sigma ) - \nabla f ( \rho_k ) }_\infty \sqrt{2 D ( \rho_k ( r^{-1} \alpha_k ), \rho_k )} 
& \geq \norm{ \nabla f ( \sigma ) - \nabla f ( \rho_k ) }_\infty \norm{ \rho_k ( r^{-1} \alpha_k ), \rho_k }_1 \notag \\
& \geq \braket{ \nabla f ( \sigma ) - \nabla f ( \rho_k ), \rho_k ( r^{-1} \alpha_k ) - \rho_k } . \label{eq_blabla_3}
\end{align}
for large enough $k \in \mathcal{K}$.  
Notice that $r^{-1} \alpha_k \leq \bar{\alpha}$ for large enough $k \in \mathcal{K}$. 
By Lemma \ref{lem_inner_product} and Corollary \ref{cor_bertsekas_asymptotic}, we obtain
\begin{align}
- \braket{ \nabla f ( \rho_k ), \rho_k ( r^{-1} \alpha_k ) - \rho_k ( \alpha_k ) } & \geq \frac{D ( \rho_k ( r^{-1} \alpha_k ), \rho_k )}{r^{-1} \alpha_k } \notag \\
& \geq \sqrt{ \kappa D ( \rho_k ( \bar{\alpha} ), \rho_k ) } \sqrt{ D ( \rho_k ( r^{-1} \alpha_k ), \rho_k ) } , \label{eq_blabla_4}
\end{align}
for large enough $k \in \mathcal{K}$. 
Since $D ( \rho_k ( r^{-1} \alpha_k ), \rho_k )$ is strictly positive for all $k \in \mathcal{K}'$ by assumption, \eqref{eq_blabla_2}, \eqref{eq_blabla_3}, and \eqref{eq_blabla_4} imply
\begin{equation}
\norm{ \nabla f ( \sigma ) - \nabla f ( \rho_k ) }_\infty > ( 1 - \tau ) \sqrt{ \frac{\kappa D ( \rho_k ( \bar{\alpha} ), \rho_k )}{2} } \geq 0 . \notag
\end{equation}
Taking limits, we obtain that $D ( \rho_k ( \bar{\alpha} ), \rho_k ) \to 0$ a $k \to \infty$ in $\mathcal{K}'$. \iftoggle{qqed}{\qed}{}

\bibliographystyle{acm}

\bibliography{list}

\end{document}